\PassOptionsToPackage{dvipsnames}{xcolor}
\documentclass[a4paper,12pt]{article}

\usepackage{a4wide,amssymb,amsmath,amsthm,mathtools,thmtools, url,enumerate,float,lineno}
\usepackage{bezier,amsfonts,amssymb,graphicx,amsthm,url}
\usepackage[utf8]{inputenc}
\usepackage{caption}
\usepackage{subcaption}
\usepackage{algorithm}
\usepackage{algpseudocode}
\usepackage{amsmath}
\usepackage{amsfonts}
\usepackage{amssymb,amsthm}
\usepackage{url}
\usepackage{hyperref}
\usepackage{tikzit}
\usepackage{xparse}
\usepackage{cleveref}
\usepackage{mmacells}

\usepackage[dvipsnames]{xcolor}
\hypersetup{
	colorlinks=true,
	pdfpagemode=UseNone,
	citecolor=OliveGreen,
	linkcolor=NavyBlue,
	urlcolor=BlueGreen,
	pdfstartview=FitW
}

\newtheorem{theorem}{Theorem}
\newtheorem{definition}{Definition}
\newtheorem{proposition}{Proposition}
\newtheorem{corollary}{Corollary}
\newtheorem{lemma}{Lemma}
\newtheorem{remark}{Remark}
\newtheorem{observation}{Observation}
\newtheorem{problem}{Problem}

\newcommand{\case}[1]{\medskip\noindent\textit{Case #1.} }

\tikzstyle{vertex}=[fill=black, draw=black, shape=circle, thick, scale=0.5]
\tikzstyle{blue_clique}=[fill=white, draw=blue, shape=circle, ultra thick]
\tikzstyle{clique}=[fill=white, draw=black, shape=circle]
\tikzstyle{red_vertex}=[fill=red, draw=red, shape=circle, scale=0.75, thick]

\tikzstyle{edge}=[-, fill={rgb,255: red,0; green,128; blue,128}]
\tikzstyle{dir_edge}=[->]
\tikzstyle{blue_edge}=[-, thick, draw=blue]
\tikzstyle{red_edge}=[-, thick, draw=red]
\tikzstyle{box}=[-, fill=gray, fill opacity=0.2]
\tikzstyle{dashed_edge}=[-, dashed]
\tikzstyle{dashed_red_edge}=[-, draw=red, dashed]
\tikzstyle{dashed_blue_edge}=[-, draw=blue, dashed]
\tikzstyle{dashed_edge_green}=[-, draw=green, dashed]
\tikzstyle{green_edge}=[-, draw=green]
\tikzstyle{B}=[-, very thick, draw={rgb,255: red,31; green,119; blue,180}]
\tikzstyle{R}=[-, very thick, draw={rgb,255: red,214; green,39; blue,40}]
\tikzstyle{O}=[-, very thick, draw={rgb,255: red,255; green,127; blue,14}]
\tikzstyle{G}=[-, very thick, draw={rgb,255: red,44; green,160; blue,44}]
\tikzstyle{P}=[-, very thick, draw={rgb,255: red,148; green,103; blue,189}]
\tikzstyle{T}=[-, very thick, draw={rgb,255: red,23; green,190; blue,207}]
\tikzstyle{Bd}=[-, very thick, draw={rgb,255: red,31; green,119; blue,180}, dashed]
\tikzstyle{Od}=[-, very thick, dashed, draw={rgb,255: red,255; green,127; blue,14}]
\tikzstyle{Gd}=[-, very thick, draw={rgb,255: red,44; green,160; blue,44}, dashed]
\tikzstyle{Rd}=[-, very thick, draw={rgb,255: red,214; green,39; blue,40}, dashed]
\tikzstyle{Pd}=[-, very thick, dashed, draw={rgb,255: red,148; green,103; blue,189}]
\tikzstyle{Td}=[-, very thick, draw={rgb,255: red,23; green,190; blue,207}, dashed]

 \tikzset{every picture/.style={font issue=\footnotesize},
          font issue/.style={execute at begin picture={#1\selectfont}}
         }

\newcommand{\ex}{\mathrm{ex}}
\newcommand{\KMP}{\mathcal{K}_{\mathrm{MP}}}
\newcommand{\KMOP}{\mathcal{K}_{\mathrm{MOP}}}
\newcommand{\K}{\mathcal{K}}
\newcommand{\MP}[2]{r_{\mathrm{MP}}\left(#1, #2\right)}
\newcommand{\Wd}[2]{\mathrm{Wd}\left(#1, #2\right)}
\newcommand{\MOP}[2]{r_{\mathrm{MOP}}\left(#1, #2\right)}

\title{Ramsey properties of maximal (outer)planar graphs}
\author{\begin{tabular}{ccc}Adriana Baldacchino & Yair Caro & Xandru Mifsud \\ {\small University of Birmingham} & {\small University of Haifa-Oranim} & {\small University of Oxford} \\ \href{mailto:axb2191@student.bham.ac.uk}{\small axb2191@student.bham.ac.uk} & \href{mailto:yacaro@kvgeva.org.il}{\small yacaro@kvgeva.org.il} & \href{mailto:xandru.mifsud@cs.ox.ac.uk}{\small xandru.mifsud@cs.ox.ac.uk}\end{tabular}}
\date{}

\DeclareGraphicsExtensions{.pdf,.png,.jpg}

\begin{document}

\maketitle

\abstract{
We study a natural extension of Ramsey theory relative to the classes of maximally planar and maximally outerplanar graphs. This can be seen as a continuation of the study of \emph{`Planar Ramsey theory'}, introduced by Axenovich et al. The question we ask is the following: For a fixed family $\mathcal{K}$ of graphs and a pair of graphs  $\{H,F\}$, does there exist an integer $r_{\mathcal{K}} (H, F)$ such that for every graph $G \in \mathcal{K}$ with $|G| \geq r_{\mathcal{K}}(H, F)$, every red/blue edge-colouring of $G$ admits a red copy of $H$ or a blue copy of $F$? When such an integer exists, we say $\{H,F\}$ is \emph{unavoidable in $\mathcal{K}$}, and otherwise $\{H,F\}$  is \emph{avoidable in $\mathcal{K}$}.

Our work focuses on this problem where  $\mathcal{K} = \mathcal{K}_{\mathrm{MOP}}$ and $\mathcal{K} = \mathcal{K}_{\mathrm{MP}}$, which denote the families of maximal outerplanar (MOP) graphs and maximal planar (MP) graphs, respectively. This framework generalises the classical Ramsey problem relative to these classes, as the case with $\mathcal{K} = \{K_n \colon n \geq 2\}$ corresponds to classical Ramsey. We also study the corresponding Ramsey numbers for MOP and MP, which we denote as $r_{\mathrm{MOP}}(H, F)$ and $r_{\mathrm{MP}}(H, F)$. 

 In the case when $\mathcal{K} = \mathcal{K}_{\mathrm{MOP}}$, we completely determine all unavoidable pairs $\{H, F\}$ with $|E(F)| \geq 2$, together with upper bounds and sometimes exact values of $r_{\mathrm{MOP}}(H, F)$. 
 When $\mathcal{K} = \mathcal{K}_{\mathrm{MP}}$, we completely determine all unavoidable pairs in the diagonal case $\{H, H\}$ when $H$ is connected, showing that $H$ must be one of the graphs $P_3$, $P_4$, $P_5$, $K_{1, 3}$ or the fork graph $S_{2,1,1}$.

This work opens up further possibilities in the study of Ramsey theory relative to a class, and we offer several open problems in this vein.
}

\section{Introduction}

Planar variants of Ramsey theory have garnered significant attention throughout the years \cite{AxenovichSchadeThomassenUeckerdt2019, goetze2020pnfree,MaTangYu2020, SteinbergTovey1993}. Our work is inspired by recent work of Axenovich et al.\ on planar Ramsey theory, in which they study graphs which are considered Ramsey relative to planar graphs \cite{AxenovichSchadeThomassenUeckerdt2019}. This work was further developed in~\cite{goetze2020pnfree,MaTangYu2020, DBLP:journals/ajc/BiroW22}, and in particular~\cite{goetze2020pnfree} investigated  Ramsey relative to outerplanar graphs. 
We study a variant of planar Ramsey for maximal outerplanar and outerplanar graphs, which differs slightly from the variant present in~\cite{AxenovichSchadeThomassenUeckerdt2019,MaTangYu2020,goetze2020pnfree}.
In order to proceed, we first clarify our notion of Ramsey relative to a class, by defining the $\mathcal{K}$-unavoidable pairs of graphs.

\begin{definition}[(Un)avoidable pair] \label{def:umpp}
Given graphs $\{H, F\}$, and a class $\mathcal{K}$, we say that  the pair $\{H,F\}$ is \emph{$\mathcal{K}$-unavoidable} if there exists an $N$ such that for every $G \in \mathcal{K}$ of order at least $N$, every red/blue edge-colouring of $G$ contains a red copy of $H$ or blue copy of $F$.  Furthermore, $\{H,F\}$ is \emph{$\mathcal{K}$-avoidable} if $\{H,F\}$ is not $\mathcal{K}$-unavoidable.
\end{definition}

This definition is well within the spirit of Ramsey theory restricted to families of graphs, as it reduces to the classical Ramsey case when $\mathcal{K} = \{K_n \colon n \in \mathbb{N}\}$.  Our work focuses on two instances of this problem for maximal planar graphs $\mathcal{K} = \KMP$ and maximal outerplanar graphs $\mathcal{K} = \KMOP$. We use $\MOP{H}{F}, \MP{H}{F}$  to denote the corresponding Ramsey numbers, respectively.

In this paper, we fully characterise the pairs $\{H, F\}$ of graphs unavoidable in $\KMOP$, as summarised in the following theorem, whilst also providing (occasionally exact) bounds on the Ramsey number $\MOP{H}{F}$ for each case.

\begin{restatable}{restatabletheorem}{MOPGeneral} \label{thm:unavoidable-MOP-general}
    Let $H, F$ be graphs and let $t, t_1, t_2$ be non-negative integers. Then, $\{H,F\}$ is unavoidable (up to symmetry) in $\KMOP$ if and only if, 
	\begin{itemize}
		\item there exists $N(H) \geq |H|$ such that every maximal outerplanar graph on at least $N(H)$ vertices has a copy of $H$ and $F \simeq K_2$ \quad \emph{or}
		\item $H$ is a subgraph of $S_{2,2,1} \cup t K_2$ and $F \simeq P_3$ or $F \simeq K_2$  \quad \emph{or}
		\item $H$ is a subgraph of $P_5 \cup t_1 K_2$ and $F$ is a subgraph of $P_4 \cup t_2 K_2$.
	\end{itemize}
\end{restatable}

We also fully characterise the connected graphs $H$ for which $\{H, H\}$ is unavoidable in $\KMP$, as summarised in the following theorem, giving lower bounds on $\MP{H}{H}$ when possible and giving some preliminary results for the off-diagonal case.

\begin{restatable}{restatabletheorem}{MPDiagonal} \label{thm:diag_H_conn_mp}
    If $H$ is a connected graph such that $\{H, H\}$ is unavoidable in $\KMP$, then $H$ is either a tree on at most four vertices, the path $P_5$, or the fork graph $S_{2,1,1}$.
\end{restatable}

\subsection{Related work} \label{sec:relatedwork}

For the discussion below, we shall refer to our avoidability condition as the \textit{strong unavoidable condition}. The notion of unavoidable developed in~\cite{AxenovichSchadeThomassenUeckerdt2019,MaTangYu2020,goetze2020pnfree} is distinct: A graph $H$ is referred to as $\mathcal{K}$-unavoidable in these works if and only if there exists $G \in \mathcal{K}$ such that every red/blue edge-colouring of $G$ contains a monochromatic copy of $H$. Henceforth we will refer to this condition as the \textit{weak unavoidability condition}.
The difference here is if $G$ is a witness of the unavoidability of $H$, there may still be a graph in $\mathcal{K}$ larger than $G$ which has a red/blue edge-colouring  that does not contain $H$. Another point of difference to note is that in our case, we consider pairs of unavoidable graphs, while previous work focuses only on the diagonal case. However, the definitions coincide for the classical Ramsey case with $\mathcal{K} = \{K_n : n\in \mathbb{N}\}$.

Some connections can be made between the two definitions, as in particular if a graph is strongly unavoidable for a class $\mathcal{K}$, it is also weakly unavoidable. Moreover, if $H$ is weakly unavoidable with $G$ as the witness graph, then every graph in the class containing $G$ is also a witness of $H$ being weakly unavoidable. However, this is by no means forces that every sufficiently large graph is a witness of $H$ strongly unavoidable. In fact the two definitions do not coincide for the classes of planar, outerplanar and maximal outerplanar graphs. A clear example of this distinction arises in considering the star $K_{1, k}$. For every $k \geq 2$, the star $K_{1, k}$ is weakly unavoidable for the class of maximal planar graphs, while as we shall prove later, $\{K_{1,4} , K_{1,4}\}$ is avoidable in maximal planar graphs under the stronger unavoidability model. 

The differences between the two models translates to a difference in proofs. To show that a graph $H$ is weakly unavoidable, one can simply provide a witness for this fact, that is, one provides a graph $G\in \mathcal{K}$ such that every red/blue edge-colouring of $G$ contains a monochromatic copy of $H$. Showing avoidability is significantly harder, as one needs to prove that for every graph $G$, there exists a red/blue edge-colouring of $G$ that does not contain a copy of $H$. The difficulty flips in our case. Showing avoidability of a graph $H$ under the strong unavoidability model requires supplying an infinite set of positive integers $S$ and family of graphs $(G_i)_{i\in S}$ from $\mathcal{K}$ such that, for every $i \in S$, there exists a red/blue edge-colouring of $G_i$ that does not contain a monochromatic copy of $H$. On the other hand, for strong unavoidability we need a proof that holds for all sufficiently large graphs in our class and for every red/blue edge-colouring of these graphs.

In our formulation, we also study the Ramsey number relative to a class $\mathcal{K}$, denoted by $r_{\mathcal{K}}(H,F)$, which is the least integer required such that every $G \in \mathcal{K}$ of order at least $r_{\mathcal{K}}(H,F)$, every red/blue edge-colouring of $G$ contains a red copy of $H$ or a blue copy of $F$.  
A notion of planar Ramsey numbers for the weakly unavoidable case has also been studied in \cite{goetze2020pnfree}. Because of the difference in formulation, our notion of Ramsey number is stronger as it has implications for every large enough graph in our class, while the formulation in \cite{goetze2020pnfree} only implies that there exists a graph of order $N$ which is Ramsey for the particular graph. Another variant of ‘planar Ramsey numbers’ has been studied in \cite{SteinbergTovey1993}, where the authors calculate the least $N$ such that every planar graph or its complement contains a particular graph. Both these notions are distinct from ours, however they highlight an interest in these kinds of problems.

Another related line of work is in the study of planar Tur{\'a}n, first studied in \cite{first-planar-turan}. The question investigated here is as follows: Given a positive integer $n$ and a graph $H$, what is the maximum number of edges in an $H$-free planar graph on $n$ vertices? This number is referred to as the \textit{planar Tur{\'a}n number} and denoted by  $\ex_{\mathcal{P}}(n, H)$. Planar Tur{\'a}n numbers have been extensively studied, see \cite{GLZ, SHI2025104134} and references therein. 

A variant of these numbers has also been studied for outerplanar graphs, where we ask given a positive integer $n$ and a graph $H$, what the maximum number of edges is in an $H$-free \emph{outer}planar graph on $n$ vertices, denoted as $\ex_{\mathcal{OP}}(n, H)$.  Various work has been carried out in this area, including outerplanar numbers of paths and cycles \cite{cycles-pahts}, disjoint copies of paths \cite{disjointcopies} and  double stars \cite{outerplanarturannumberdoublestars}.

Our work is closely related to these Tur{\'a}n numbers, as given we study maximal planar and maximal outerplanar graphs, our graphs on $n$ vertices have exactly  $3n-6$ and $2n-3$ edges. This allows us to utilise results on planar Tur{\'a}n numbers in our work. Furthermore, if we were to restrict our definition to use only one colour, the unavoidable graphs are exactly those for which the planar/outerplanar Tur{\'a}n number is less than $3n-6$ and $2n-3$ respectively. Therefore, the two questions are closely related, where investigations in (outer)planar Tur{\'a}n numbers can be seen as the density counterpart of our maximal (outer)planar Ramsey problem.

\subsection{Notation}

We next establish our working notation. Given a graph $G$, by $V(G)$, $E(G)$, $\delta(G)$ and $\Delta(G)$ we denote the vertex set, edge set, minimum degree and maximum degree of $G$, respectively. When there is no room for ambiguity, we shall omit any reference to $G$ in our notation and simply write $V$, $E$, $\delta$ and $\Delta$.

Given an integer $n \geq 1$, by $P_n$ we denote the path on $n$ vertices. A \textit{linear forest} is a forest where every component is a path. For $n \geq 3$, by $C_n$ we denote the cycle on $n$ vertices and by $W_n$ we denote the \textit{wheel graph}, obtained by taking a a cycle $v_1, \dots, v_n$ and adding a vertex $c$ adjacent to each $v_i$ for $1 \leq i \leq n$. 

Also, given integers $n, m \geq 1$, by $K_n$ we denote the complete graph on $n$ vertices, and by $K_{n, m}$ we denote the complete bipartite graph with partite sets of sizes $n$ and $m$. By $\Wd{n}{m}$ we denote the \textit{windmill graph}, constructed by coalescing $m$ copies of $K_n$ at a vertex. Lastly, given integers $p \geq q \geq r \geq 1$, by $S_{p,q,r}$ we denote the \textit{spider graph} on $p + q + r + 1$ vertices, constructed by coalescing three paths $P_{p+1}, P_{q+1}$ and $P_{r+1}$ at a leaf. The \textit{fork graph} is the spider $S_{2,1,1}$.

For any two graphs $G$ and $H$, by $G \cup H$ we denote the disjoint union of $G$ and $H$, having vertex set $V(G) \cup V(H)$ and edge set $E(G) \cup E(H)$. Furthermore, given an integer $k \geq 2$, we use $k G$ as short-hand for the disjoint union of $k$ copies of $G$.

The \textit{join} $G + H$ of $G$ and $H$ is $G \cup H$ together with all edges between $V(G)$ and $V(H)$. Given positive integers $m$ and $n$, by $F_{m, n}$ we denote the \textit{fan graph} $m K_1 + P_n$.

\subsection{Organisation of paper}

In Section \ref{sec:preliminaries} we will give a number of useful observations and techniques which will be used throughout the paper. In Section \ref{sec:mop} we will prove Theorem \ref{thm:unavoidable-MOP-general}, and in Section \ref{sec:mp} we will prove Theorem \ref{thm:diag_H_conn_mp}. Throughout we will also determine bounds and occasionally exact values on the Ramsey number of a given unavoidable pair in $\KMOP$ and/or $\KMP$. For several such instances, we will employ a combination of proof and exhaustive computer searches to check whether a given pair is (un)avoidable in small maximal (outer)planar graphs; details may be found in Appendix \ref{sec:code}.

\section{Basic observations and technical lemmas} \label{sec:preliminaries}

In this section we will give a number of useful observations and techniques which we will use throughout. We begin by noting the following so-called \textit{monotonicity} and \textit{symmetry} properties.

\begin{observation}[Monotonicity] \label{obs:mop_monotonicity}
 Let $\K \in \{\KMOP, \KMP\}$. Let $H, F$ be graphs. Let $H^*$ and $F^*$ be subgraphs of $H$ and $F$, respectively. 

 \begin{enumerate}[(i)]
 \item If $\{H^*, F^*\}$ is avoidable in $\K$, then $\{H, F\}$ is avoidable in $\K$.
 \item If $\{H, F\}$ is unavoidable in $\K$, then $\{H^*, F^*\}$ is unavoidable in $\K$.
 \end{enumerate}
\end{observation}

\begin{observation}[Symmetry] \label{obs:symmetry}
 Let $\K \in \{\KMOP, \KMP\}$. Let $H, F$ be graphs. Then, $\{H, F\}$ is unavoidable in $\K$ if, and only if, $\{F, H\}$ is unavoidable in $\K$.
\end{observation}

The following observation establishes that it suffices to consider pairs $\{H, F\}$ where every component has order at least $2$. Henceforth we will assume that all our graphs do not have isolated vertices.

\begin{observation}[Isolated vertices] \label{obs:mop_isolated_vertices}
	Let $\K \in \{\KMOP, \KMP\}$. Let $H, F$ be graphs and let $s, t$ be non-negative integers. Then, $\{H, F\}$ is unavoidable in $\K$ if, and only if, $\{H \cup t K_1, F \cup s K_1\}$ is unavoidable in $\K$.
\end{observation}

The most `primitive' pairs are of the form $\{H, K_2\}$, which are unavoidable in $\K$ if, and only if, there exists an integer $N(H) \geq |H|$ such that every graph $G$ in $\K$ with $|G| \geq N(H)$ has a subgraph isomorphic to $H$ (noting that every red/blue edge-colouring of $G$ is either monochromatic using red or has at least one blue edge).

\begin{observation}[Primitive $H$] \label{obs:H_K2}
 Let $\K \in \{\KMOP, \KMP\}$ and let $H$ be a graph. Then, $\{H, K_2\}$ is unavoidable in $\K$ if, and only if, there exists an integer $N(H) \geq |H|$ such that every graph $G$ in $\K$ with $|G| \geq N(H)$ has a subgraph isomorphic to $H$.
\end{observation}

Lastly, we present the following technical result, which we term as the `rim-trick', that will enable us to construct unavoidable pairs from already known unavoidable pairs. First, we require the following lemma on paths.

\begin{lemma}\label{lem:interval-lemma}
 Let $k, n$ be two integers such that $n \geq k \geq 1$. Every red/blue edge-colouring of the path $P_{n}$ either has a blue $P_{k}$ or at least $\lfloor (n-1)/(k-1) \rfloor$ red edges.
\end{lemma}

\begin{proof}
 Suppose the path $P_{n}$ has vertices $v_0, \dots, v_{n-1}$ and edges $e_i \coloneqq \{v_i, v_{i+1}\}$ for all $0 \leq i < n-1$. 
 Suppose we have a red/blue edge-colouring of $P_{n}$, such that there is no blue $P_{k}$. Then, the collections of edges $\{e_0, \dots, e_{k-2}\}, \{e_{(k-1)},\dots, e_{2(k-1)-1} \}, \dots \{\allowbreak e_{(\lceil (n-1)/(k-1) \rceil -1)(k-1)}, \dots ,e_{(\lceil (n-1)/(k-1) \rceil)(k-1) -1} \}$ each have at least one red edge, as otherwise we would have a blue $P_{k}$. Therefore we must have at least $\lfloor (n-1)/(k-1) \rfloor$ red edges in our colouring, as desired.
\end{proof}

\newcommand{\rimbd}{C}
\begin{theorem}[Rim-trick]\label{thm:rim-trick-cycle}
 Let $H$ be a graph and $F$ a linear forest. If $G$ is a graph in which $\{H, F\}$ is unavoidable, and contains a cycle of length $\rimbd \geq (2t-1+|H|)(|F|-1) +2|H|$, then $\{ H\cup tK_2, F\}$ is unavoidable in $G$.
\end{theorem}

\begin{proof}
 Consider a red/blue edge-colouring of $G$, and suppose that this colouring does not contain a blue $F$. Because $\{H, F\}$ is unavoidable, this means $G$ contains a red $H$. Consider the graph $G^*$ obtained from $G$ by deleting the vertices of this copy of $H$. In this $G^*$, the cycle of length $\rimbd$ is split into $p$ paths, $R_1,\dots, R_{p}$ where $p \leq |H|$ which in total have $N \geq \rimbd - |H|$ vertices. Let $r, b$ represent the number of red and blue edges in these paths respectively. If $r \geq 2t-1$, then since forests are Vizing Class 1 graphs, there exists a proper $2$-edge-colouring of the red edges of these paths, and therefore there is a matching of size $t$ amongst the red edges of these paths, and hence in $G$. Together with the red $H$, this gives us a red $H \cup tK_2$.

 Otherwise, it must be the case that $r < 2t-1$. 
 Since $r+b = N - p$, $b< N-p -2t +1$. 
 Since our colouring does not contain a blue $F$, it must be the case that none of the $p$ paths contain a path with $|F|$ vertices, as $F$ can be embedded within this path since it is a linear forest. Letting $n_i$ represent the number of vertices and $r_i$ the number of red edges in path $R_i$, by Lemma \ref{lem:interval-lemma} it must be the case that $r_i \geq \lfloor (n_i-1)/(|F|-1) \rfloor$. Therefore, 
 $r \geq \sum_{i=0}^{p} \lfloor (n_i-1)/(w-1) \rfloor >
 \sum_{i=0}^{p} (n_i-1)/(|F|-1) -1 \geq (N-p)/(|F|-1) -p \geq ( N- |H|)/(|F|-1) - |H| \geq ( C- 2|H|)/(|F|-1) - |H|$. 

 To obtain a contradiction, we require that $r \geq 2t-1$, and hence it suffices to show that $ ( \rimbd - 2|H|)/(|F|-1) - |H| \geq 2t-1$. Rearranging this equation we get $\rimbd \geq (2t-1+|H|)(|F|-1) +2|H|$ as in the hypothesis, therefore we obtain our desired result.
\end{proof}

\begin{remark}
 If instead of a cycle we have a path, we can also carry out this trick with the slightly worse bound of $C \geq (2t+|H|)(|F| - 1) +2|H|+1$, noting that when splitting the path we get at most $|H|+1$ paths.
\end{remark}
\section{(Un)avoidable monochromatic pairs in \texorpdfstring{$\KMOP$}{KMOP}}
\label{sec:mop}

Let $H$ and $F$ be graphs, possibly with $H = F$, and recall the definition of an \textit{unavoidable} pair $\{H, F\}$ in $\KMOP$ from Definition \ref{def:umpp}. Recall also that if $\{H,F\}$ is unavoidable, we write $\MOP{H}{F}$ for the smallest positive integer $N$ such that every red/blue edge-colouring of every maximal planar graph on at least $N$ vertices contains either a red copy of $H$ or a blue copy of $F$. We begin by summarising a number of observations. 

Before proceeding, we note the following facts about maximal outerplanar graphs. 

\begin{theorem}[\cite{Harary_1969} Theorem 11.9 and Corollary 11.9 (a)] \label{thm:mop_properties}
    Every maximal outerplanar graph with $n \geq 3$ vertices has $n - 2$ interior faces, $2n - 3$ edges, and at least two vertices of degree~$2$.
\end{theorem}

\begin{corollary}[Induction on maximal outerplanar graphs]\label{cor:for-ind}
    If $G$ is a maximal outerplanar graph and $x$ is a vertex of degree $2$ in $G$, then $G - x$ is also maximal outerplanar.
\end{corollary}

\begin{proof}
    Let $G$ be a maximal outerplanar graph on $n$ vertices and let $x$ be a vertex of degree $2$ in $G$; note then that $n \geq 3$. Firstly, $G - x$ must be outerplanar, as removing a vertex preserves outerplanarity. Also, by Theorem \ref{thm:mop_properties}, $G$ has $2n - 3$ edges, which means $G - x$ must have $2(n - 1) - 3$ edges. Therefore $G - x$ must be maximal, as otherwise we can add more edges whilst remaining outerplanar, contradicting Theorem \ref{thm:mop_properties}. 
\end{proof}

\subsection{Constructions for families of avoidable pairs in \texorpdfstring{$\KMOP$}{KMOP}}

We outline the following constructions which we will use throughout this section, in order to demonstrate which pairs $\{H, F\}$ are avoidable in $\KMOP$.

\begin{figure}[H]
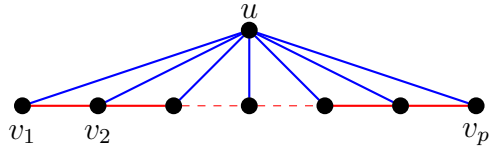

   \ctikzfig{figures/mop_fig1}
   \caption{The family of $2$-edge-coloured maximal outerplanar graphs $A^{(1)}_p$ on $p + 1$ vertices, $p \geq 1$.}
   \label{fig:mop_fig1}
\end{figure}

\begin{figure}[h!]
   \ctikzfig{figures/mop_fig5}
   \caption{The family of $2$-edge-coloured maximal outerplanar graphs $A^{(2)}_p$ on $2p + 1$ vertices, $p \geq 1$.}
   \label{fig:mop_fig2}
\end{figure}

\begin{figure}[h!]
   \ctikzfig{figures/mop_fig2}
   \caption{The family of $2$-edge-coloured maximal outerplanar graphs $A^{(3)}_p$ on $2p$ vertices, $p \geq 1$.}
   \label{fig:mop_fig3}
\end{figure}

\begin{figure}[h!]
   \ctikzfig{figures/mop_fig3}
   \caption{The family of $2$-edge-coloured maximal outerplanar graphs $A^{(4)}_p$ on $2p$ vertices, $p \geq 1$.}
   \label{fig:mop_fig4}
\end{figure}

\begin{figure}[h!]
   \ctikzfig{figures/mop_fig6}
   \caption{The family of $2$-edge-coloured maximal outerplanar graphs $A^{(5)}_p$ on $2p$ vertices, $p \geq 1$.}
   \label{fig:mop_fig5}
\end{figure}

\begin{figure}[h!]
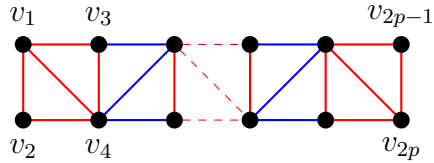

   \ctikzfig{figures/mop_fig4}
   \caption{The family of $2$-edge-coloured maximal outerplanar graphs $A^{(6)}_p$ on $2p$ vertices, $p \geq 1$.}
   \label{fig:mop_fig6}
\end{figure}

\subsection{The diagonal case: unavoidable \texorpdfstring{$\{H, H\}$}{(H,H)} in \texorpdfstring{$\KMOP$}{KMOP}}

In this section we will prove the following theorem, which gives the complete characterisation of which pairs $\{H, H\}$ are unavoidable in $\KMOP$. 

\begin{restatable}{restatabletheorem}{MOPDiagonal} \label{thm:mop_diagonal}
    $\{H, H\}$ is unavoidable in $\KMOP$ if, and only if, $H$ is isomorphic to one of $t K_2$, $P_3 \cup t K_2$ or $P_4 \cup t K_2$ for some integer $t \geq 0$.
\end{restatable}

We will first consider the case when $H$ is connected. We will require the following two lemmas.

\begin{lemma} \label{lem:mop_f14}
    Every maximal outerplanar graph on $n \geq 5$ vertices has a vertex-induced subgraph isomorphic to the fan $F_{1, 4}$. 
\end{lemma}

\begin{proof}
    We begin by noting that there is only one maximal outerplanar graph on five vertices, namely the fan $F_{1, 4}$. We proceed by induction on the number of vertices $n \geq 6$. Suppose the statement holds for all maximal outerplanar graphs on $n - 1$ vertices. Let $G$ be a maximal outerplanar graph on $n$ vertices. By Theorem \ref{thm:mop_properties} and Corollary \ref{cor:for-ind}, there exists a vertex $v$ of degree $2$  such that $G - v$ is a maximal outerplanar graph on $n - 1$ vertices. By the inductive hypothesis, $G - v$ has a vertex-induced subgraph isomorphic to $F_{1, 4}$. The result follows.
\end{proof}

\begin{lemma} \label{lem:mop_p4}
    $\{P_4, P_4\}$ is unavoidable in $\KMOP$, with $\MOP{P_4}{P_4} = 5$.
\end{lemma}

\begin{proof}
    The only maximal outerplanar graph on $4$ vertices is the fan $F_{1, 3}$, which admits a red/blue edge-colouring such that there is no monochromatic $P_4$. Namely, consider the colouring where we fix a vertex $v$ of degree $3$ and colour all edges incident to $v$ red, whilst colouring the remaining edges blue. Then such a colouring has no monochromatic $P_4$.
    
    Hence $\MOP{P_4}{P_4} \geq 5$. We next show equality. There is only one maximal outerplanar graph with $5$ vertices, the fan $F_{1, 4}$. It is easy to check that for all red/blue edge-colourings of this graph, $P_4$ is embeddable in either the blue or red subgraphs. The fact that $\{P_4, P_4\}$ is unavoidable for any maximal outerplanar graph on $n \geq 5$ vertices then follows by Lemma \ref{lem:mop_f14}.
\end{proof}

\begin{proposition} \label{prop:mop_diagonal_connected}
    If $H$ is a connected graph, then $\{H, H\}$ is unavoidable 
    $\KMOP$ if, and only if, $H$ is isomorphic to one of $K_2, P_3$ or $P_4$.
\end{proposition}

\begin{proof}
    Suppose that $\{H, H\}$ is unavoidable in $\KMOP$. By the construction in Figure \ref{fig:mop_fig6}, $H$ must have at most $4$ vertices. Furthermore, by the construction in Figure \ref{fig:mop_fig1}, $H$ cannot contain any cycles. Also, by the construction in Figure \ref{fig:mop_fig3}, $H$ cannot be the star $K_{1, 3}$. Hence $H$ could only possibly be isomorphic to one of $K_2$, $P_3$ or $P_4$. Trivially, $\{K_2, K_2\}$ are unavoidable in $\KMOP$. By Lemma \ref{lem:mop_p4}, $\{P_4, P_4\}$ is unavoidable in $\KMOP$, and by monotonicity it follows that $\{P_3, P_3\}$ is also unavoidable in $\KMOP$. The result follows. 
\end{proof}

We next strengthen the above result using the `rim-trick' (Theorem \ref{thm:rim-trick-cycle}), in order to obtain results for when $H$ is not necessarily connected. 

\begin{theorem} \label{thm:mop_p4_matching}   
    For every integer $t \geq 0$, $\{P_4 \cup t K_2, P_4 \cup t K_2\}$ is unavoidable in $\KMOP$ and $\MOP{P_4 \cup t K_2}{ P_4 \cup t K_2} \leq 4t^2 +12t + 17$.
\end{theorem}
\begin{proof}
    We first prove that $\{P_4 \cup t K_2, P_4 \}$ is unavoidable in $\KMOP$. 
    Let $G$ be a maximal outerplanar graph on $n \geq 6t+17$ vertices. Since $G$ is Hamiltonian, it contains a cycle of length $\geq 6t+17 = (2t-1+|P_4|)(|P_4|-1) +2|P_4|$. Moreover, as $P_4$ is a linear forest 
    and $G$ either has a red $P_4$ or a blue $P_4$ by Lemma \ref{lem:mop_p4}, this means we can apply Theorem \ref{thm:rim-trick-cycle} to conclude that every red/blue edge-colouring of $G$ graph has a red $P_4 \cup t K_2$ or a blue $P_4$.

    Now, let $H$ be a maximal outerplanar graph on $n' \geq 4t^2 +12t + 17 > 6t+17$  vertices. By the previous result and flipping the colours, $H$ either has a red $P_4$ or a blue $P_4 \cup t K_2$. $P_4 \cup t K_2$ is a linear forest, and since $H$ is Hamiltonian, $H$ contains a cycle of length  $\geq 4t^2 +12t + 17 = (2t-1+|P_4|)(|P_4\cup tK_2|-1) +2|P_4|$ therefore we can apply Theorem \ref{thm:rim-trick-cycle} to obtain that every red/blue edge-colouring of $H$ has a red $P_4 \cup t K_2$ or a blue $P_4 \cup t K_2$ as desired.
\end{proof}

We are now in a position to prove Theorem \ref{thm:mop_diagonal}, which we re-state for convenience. 

\MOPDiagonal*

\begin{proof}
    If $\{H, H\}$ is unavoidable in $\KMOP$ then, by monotonicity, for every component $H^*$ of $H$ the pair $\{H^*, H^*\}$ must be unavoidable in $\KMOP$. Therefore, by Proposition \ref{prop:mop_diagonal_connected}, all the components of $H$ must isomorphic to $K_2$, $P_3$ or $P_4$. By the construction in Figure \ref{fig:mop_fig2}, $H$ cannot have more than one component on three or more vertices. Therefore $H$ must be isomorphic to one of $t K_2$, $P_3 \cup t K_2$ or $P_4 \cup t K_2$ for some integer~$t \geq 0$. 

    By Theorem \ref{thm:mop_p4_matching}, $\{P_4 \cup t K_2, P_4 \cup t K_2\}$ is unavoidable in $\KMOP$, and by monotonicity so are $\{t K_2, t K_2\}$ and $\{P_3 \cup t K_2, P_3 \cup t K_2\}$.
\end{proof}

\subsection{The off-diagonal case: unavoidable \texorpdfstring{$\{H, F\}$}{(H,F)} in \texorpdfstring{$\KMOP$}{KMOP}}

In this section we extend Theorem \ref{thm:mop_diagonal} for the diagonal case, in order to give the full characterisation in Theorem \ref{thm:unavoidable-MOP-general} of all pairs $\{H, F\}$ which are unavoidable in $\KMOP$.

We proceed in our characterisation of all unavoidable pairs $\{H, F\}$ in $\KMOP$ by cases on the value of $\Delta(H)$.

\begin{lemma}\label{lem:delta=5}
    Let $H, F$ be graphs. If $\Delta (H) \geq 5$, then $\{H, F\}$ is avoidable in $\KMOP$.
\end{lemma}

\begin{proof}
    First note that there exists an infinite family of maximal outerplanar graphs $\{G_i \colon i \geq 1\}$ with $\Delta(G_i) = 4$ (e.g. consider the family in Figures \ref{fig:mop_fig3} and \ref{fig:mop_fig4}, without the edge-colouring). Colouring the edges of each $G_i$ red, $H$ is not a subgraph of any $G_i$ since $\Delta(H) \geq 5$. Also, since there are no blue edges, there is no blue $F$ either. Hence $\{H, F\}$ is avoidable in $\KMOP$.
\end{proof}

\begin{proposition}\label{prop:drawing-charact}
    Let $H, F$ be graphs. Then,
    \begin{enumerate}[(i)]
        \item if $\Delta(H) = 3$ and $\{H, F\}$ is unavoidable in $\KMOP$, then $F \simeq K_2$ or $F \simeq P_3$;
        \item if $\Delta(H) = 4$ and $\{H, F\}$ is unavoidable in $\KMOP$, then $F \simeq K_2$.
    \end{enumerate}
\end{proposition}

\begin{proof}
    First note that by the construction in Figure \ref{fig:mop_fig1}, for $\Delta(H) \geq 3$ there cannot be a red copy of $H$, and therefore there must be a blue copy of $F$, i.e. $F$ must be a star. Similarly, by the construction in Figure \ref{fig:mop_fig3}, $\Delta(F) \leq 2$. Therefore $F$ must be isomorphic to either $K_2$ or $P_3$. In particular, if $\Delta(H) = 4$ then by the construction in Figure \ref{fig:mop_fig4}, $F$ cannot be isomorphic to $P_3$ and therefore in this case $F$ must be isomorphic to $K_2$.
\end{proof}

The following theorem is an immediate consequence of Proposition \ref{prop:drawing-charact} (ii) and Observation \ref{obs:H_K2}.

\begin{theorem}\label{thm:delta=4}
    Let $H, F$ be two graphs such that $\Delta(H) = 4$. $\{H, F\}$ is unavoidable in $\KMOP$ if, and only if, $F \simeq K_2$ and there exists an integer $N(H) \geq |H|$ such that every maximal outerplanar graph on $N(H)$ or more vertices contains a subgraph isomorphic to $H$.
\end{theorem}

\begin{remark} \label{rem:mop_H_K2_problem}
    The graphs for which there exists an integer $N(H) \geq |H|$ such that every maximal outerplanar graph on $N(H)$ or more vertices contains a subgraph isomorphic to $H$ are exactly those graphs for which $\ex_{\mathcal{OP}}(n, H) < 2n-3$ for $n$ large enough. We do not fully characterise these graphs, but results on outerplanar Tur{\'a}n numbers already provide some insights on which graphs are and are not in this class.
\end{remark}

\begin{lemma}\label{lem:delta=3-subgraph}
    If $H$ is a graph such that $\Delta(H) = 3$ and $\{H, P_3\}$ is unavoidable in $\KMOP$, then $H$ must be a subgraph of $S_{2,2,1} \cup tK_2$ for some integer $t \geq 0$.
\end{lemma}

\begin{proof}
    By the construction in Figure \ref{fig:mop_fig2}, since there is no blue $P_3$, it follows that $H$ must be a subgraph of the windmill graph $\Wd{3}{m}$ for some integer $m \geq 1$. In particular, this means that $H$ can only have one component on $3$ or more vertices.
    
    Furthermore, by the construction in Figure \ref{fig:mop_fig4}, $H$ cannot have any cycle of length $3$. Since $H$ has $\Delta(H) = 3$, must be a subgraph of $\Wd{3}{m}$ for some integer $m \geq 1$ and cannot contain a cycle of length $3$, it follows that $H$ must be a subgraph of $S_{2,2,2} \cup t K_2$ for some integer $t \geq 0$.

    Lastly, by the construction in Figure $\ref{fig:mop_fig5}$, $H$ cannot have a component isomorphic to $S_{2,2,2}$, hence $H$ must be a subgraph of $S_{2,2,1} \cup t K_2$ for some integer $t \geq 0$.
\end{proof}

\begin{lemma} \label{lem:S221unav}
    $\{S_{2,2,1}, P_3\}$ is unavoidable in $\KMOP$ and  $\MOP{S_{2,2,1}}{P_3} = 7$. 
\end{lemma}
\begin{proof}
	 Figure \ref{fig:appendixB1} in Appendix \ref{sec:small_cexs} illustrates a maximal outerplanar graph on $n = 6$ vertices with no red $S_{2,2,1}$ or a blue $P_3$. 
     
     By an exhaustive computer search, as outlined in Appendix \ref{sec:code}, one can verify that $n = 7$ is the smallest $n$ such that every red/blue edge-colouring of every maximal outerplanar graph on $n$ vertices contains a red $S_{2,2,1}$ or a blue $P_3$.
     
     We then proceed by induction on the number of vertices to prove that every red/blue edge-colouring of a maximal outerplanar graph on $n \geq 7$ vertices has a red $S_{2,2,1}$ or a blue $P_3$. The base case is already solved, and the inductive case follows by Corollary \ref{cor:for-ind}.
\end{proof}

\begin{theorem}\label{thm:delta=3-unavoidable}
    $\{S_{2,2,1}\cup t K_2, P_3\}$ is unavoidable in $\KMOP$ and  $\MOP{S_{2,2,1} \cup t K_2}{P_3} \leq 4t +22$. 
\end{theorem}

\begin{proof}
    Let $G$ be a maximally outerplanar graph on $n \geq 4t +22$ vertices. First note that since $4t+22 \geq 7$, every red/blue edge-colouring of $G$ has a red $S_{2,2,1}$ or a blue $P_3$ by Lemma \ref{lem:S221unav}. Hence we can apply Theorem \ref{thm:rim-trick-cycle}, noting that $4t +22 = (2t-1+|S_{2,2,1}|)(|P_3| - 1) +2|S_{2,2,1}|$ and $P_3$ is a linear forest. This in turn gives us that every red/blue edge-colouring of $G$ must have a red $S_{2,2,1}\cup t K_2$ or a blue $P_3$ as required.
\end{proof}

\begin{lemma}\label{lem:P5P4-unav}
    $\{P_5,P_4\}$ is unavoidable in $\KMOP$ and  $\MOP{P_5}{P_4} = 9$.    
\end{lemma}
\begin{proof} 
    Figure \ref{fig:appendixB2} in Appendix \ref{sec:small_cexs} illustrates a maximal outerplanar graph on $n = 8$ vertices with no red $P_5$ or a blue $P_4$.
    
	By an exhaustive computer search, as outlined in Appendix \ref{sec:code}, one can verify that $n = 9$ is the smallest $n$ such that every red/blue edge-colouring of every maximal outerplanar graph on $n$ vertices contains a red $P_5$ or a blue $P_4$. Furthermore, we can apply Corollary \ref{cor:for-ind} to conclude by induction that every maximal outerplanar graph on $n \geq 9$ vertices has this property.
\end{proof}

\begin{theorem}\label{thm:delta=2-p5p4-unav}
    $\{P_5\cup t_1 K_2, P_4 \cup t_2 K_2\}$ is unavoidable in $\KMOP$ and \[\MOP{P_5 \cup t_1 K_2}{P_4 \cup t_2 K_2}
    \leq 4 t_1 t_2 + 8t_2 + 6 t_1 + 20.\]
\end{theorem}
\begin{proof}
    We obtain these results by applying the Theorem \ref{thm:rim-trick-cycle} twice. The two different bounds are obtained by different orders of applying the trick. Let $G$ be a maximal outerplanar graph on $n \geq 4 t_1 t_2 + 8t_2 + 6 t_1 + 20$ vertices. Firstly, note that if $t_1,t_2$ are both zero, $\MOP{P_5}{P_4} = 9$ so we can assume that  $\max(t_1,t_2)\neq 0$. With this assumption, we have that $4 t_1 t_2 + 8t_2 + 6 t_1 + 20 \geq 6t_1 + 22 = (2t_1 -1 + |P_5|)(|P_4| -1)+ 2|P_5|$. Moreover by Lemma \ref{lem:P5P4-unav}, $\{P_5, P_4\}$ is unavoidable in $G$, therefore we can apply Theorem \ref{thm:rim-trick-cycle} to obtain that $\{P_5\cup t_1 K_2, P_4\}$ is unavoidable in $G$. 
    
    Finally, since $P_5\cup t_1 K_2$ is a linear forest, and $4 t_1 t_2 + 8t_2 + 6 t_1 + 20  = (2t_2-1+|P_4|)(|P_5\cup t_1 K_2| - 1)+2|P_4|$, we can apply Theorem \ref{thm:rim-trick-cycle} again (with $H,F$ flipped) to obtain that $\{P_5\cup t_1 K_2, P_4\cup t_2 K_2\}$ is unavoidable in $G$.
\end{proof}

We are now in a position to prove Theorem \ref{thm:unavoidable-MOP-general}, which we restate for convenience. 

\MOPGeneral*

\begin{proof}
    Suppose that $\{H, F\}$ is unavoidable in $\KMOP$.
    First note that $\Delta(H), \Delta(F) \leq 4$ by Lemma \ref{lem:delta=5}. If one of $H, F$ has maximum degree $4$, then without loss of generality let $H$ be the graph of maximum degree $4$ and let $F$ be the other graph. By Proposition \ref{prop:drawing-charact}, $F$ must be $K_2$ and by Theorem \ref{thm:delta=4}, $H$ must be unavoidable in $\KMOP$.
    
    If one of $H, F$ has maximum degree $3$, then without loss of generality let $H$ be the graph of maximum degree $3$ and let $F$ be the other graph. By Proposition \ref{prop:drawing-charact}, $F$ is either $K_2$ or $P_3$, and by Lemma \ref{lem:delta=3-subgraph} $H$ is a subgraph of $S_{2,2,1} \cup tK_2$. 
    
    Finally, if both $H,F$ have degree $2$,
    without loss of generality, suppose the order of the largest component of $H$ is greater than or equal to the order of the largest component of $F$. By Figure \ref{fig:mop_fig6}, $H$ has at most a largest component of order $5$, and $F$ has at most a largest component of order $4$. Moreover by Figure \ref{fig:mop_fig2}, the other components must all be $K_2$s. Therefore $H$ is a subgraph of $P_5 \cup t_1 K_2$ and $F$ is a subgraph of $P_4 \cup t_2 K_2$ as required.

    Conversely, by Theorems \ref{thm:delta=4}, \ref{thm:delta=3-unavoidable} and \ref{thm:delta=2-p5p4-unav} for the pairs of graphs $\{H, F\}$ specified, $\{H, F\}$ is indeed unavoidable in $\KMOP$. 
\end{proof}

\section{(Un)avoidable pairs in maximal planar graphs}
\label{sec:mp}

Let $H$ and $F$ be graphs, possibly with $H = F$, and recall the definition of an \textit{unavoidable} pair $\{H, F\}$ in $\KMP$ from Definition \ref{def:umpp}. Recall also that if $\{H,F\}$ is unavoidable, we write $\MP{H}{F}$ for the smallest positive integer $N$ such that every red/blue edge-colouring of every maximal planar graph on at least $N$ vertices contains either a red copy of $H$ or a blue copy of $F$. We begin by summarising a number of observations. 

Given a positive integer $n$ and a graph $H$, the \textit{planar Tur{\'a}n number} $\ex_{\mathcal{P}}(n, H)$ is the maximum number of edges in an $H$-free planar graph on $n$ vertices. Planar Tur{\'a}n numbers have been extensively studied, see \cite{GLZ, SHI2025104134} and references therein. 

Let $n$ be a sufficiently large integer and consider a maximal planar graph $G$ on $n$ vertices with a red/blue edge-colouring. Suppose that there is no red copy of $H$ in this colouring; then the number of red edges must be at most $\ex_{\mathcal{P}} (n, H)$. Since $G$ is maximal planar, then $G$ has $3n - 6$ edges and therefore there must be at least $(3n - 6) - \ex_{\mathcal{P}} (n, H)$ blue edges. If for all $n$ sufficiently large we have that $\ex_{\mathcal{P}}(n, H) + \ex_{\mathcal{P}}(n, F) < 3n - 6$, then it follows that there must be a blue copy of $F$ in $G$ and therefore $\{H, F\}$ in unavoidable in $\KMP$.

\begin{proposition} \label{prop:mp_turan}
    Let $H$ and $F$ be graphs such that, for all $n$ sufficiently large, $\ex_{\mathcal{P}}(n, H) + \ex_{\mathcal{P}}(n, F) < 3n - 6$. Then $\{H, F\}$ is unavoidable in $\KMP$.

    Furthermore, if $\ex_{\mathcal{P}}(n, H) < \frac{3n - 6}{2}$, then $\{H, H\}$ is unavoidable in $\KMP$.
\end{proposition}

If $H$ is a tree on $\leq 4$ vertices, then $\ex_{\mathcal{P}} (n, H) = \ex (n, H) \leq n$ for all integers $n \geq |H|$. The following corollary is an immediate consequence of Proposition \ref{prop:mp_turan}

\begin{corollary} \label{cor:mp_small_trees}
    If $H$ is a tree on at most four vertices, $\{H, H\}$ is unavoidable in $\KMP$.
\end{corollary}

On the other hand, if $H$ is a tree on five vertices, then $\ex_{\mathcal{P}} (n, H) \leq \frac{3n}{2}$ for all positive integers $n$. While Proposition \ref{prop:mp_turan} reduces some instances of our problem to the study of planar Tur{\'a}n numbers, it does not provide a characterisation of unavoidable pairs in $\KMP$. Indeed, for many graphs $H$, such as trees on five vertices, $\ex_{\mathcal P}(n, H)$ is too large to force a monochromatic copy of $H$ by a simple counting argument. Consequently, the study of $\MP{H}{F}$ involves structural questions about maximal planar graphs that are not captured by planar Tur{\'a}n numbers alone.

Lastly, we note the following well-established result on the length of the longest cycle in a maximal planar graph on $n$ vertices.

\begin{theorem}[\cite{CHEN200280}, Corollary 3.5] \label{thm:mp_longest_cycle}
    Every maximal planar graph on $n$ vertices has a cycle on at least $n^{\log_3 2}$ vertices.
\end{theorem}

We proceed by giving a number of useful constructions which we will use to show avoidable pairs in $\KMP$, and then proceed with proving a full characterisation of all pairs $\{H, H\}$, where $H$ is connected, which are unavoidable in $\KMP$.

\subsection{Constructions for families of avoidable pairs in \texorpdfstring{$\KMP$}{KMP}}

We give two constructions of infinite families of $2$-edge-coloured maximal planar graphs which will allows use to show several avoidable pairs $\{H, F\}$ in $\KMP$.

First consider the infinite family of maximal planar graphs $\{B^{(1)}_p \colon p \geq 1\}$ with a $2$-edge-colouring constructed as in Figure~\ref{fig:ck_pk_mp_infty}.

\begin{figure}[ht!]
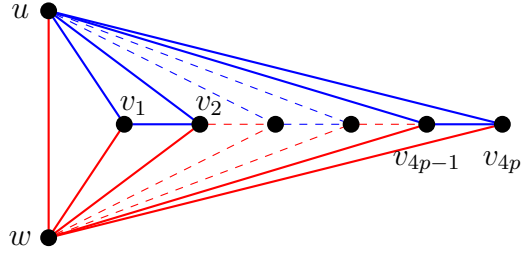

   \ctikzfig{figures/ck_pk_mp_infty}
   \caption{The family of $2$-edge-coloured maximal planar graphs $B^{(1)}_p$ on $4p + 2$ vertices, $p \geq 1$.}
   \label{fig:ck_pk_mp_infty}
\end{figure}

For every integer $p \geq 1$, neither monochromatic subgraph of $B^{(1)}_p$ contains a cycle $C_s$ for any $s \geq 4$, or $2 P_3$. Together with monotonicity, we have the following result.

\begin{lemma} \label{lem:mp_avoid_const_A}
    For all integers $s, t \geq4$, we have that $\{C_s, C_t\}$, $\{C_s, 2P_3\}$ and $\{2P_3, 2P_3\}$ are avoidable in $\KMP$. Furthermore, if $H, F$ are two graphs with a subgraph isomorphic to $2 P_3$ or $C_s$ for some $s \geq 4$, then $\{H, F\}$ is avoidable in $\KMP$.
\end{lemma}

The following remark highlights a number of graph families to which Lemma \ref{lem:mp_avoid_const_A} applies via monotonicity.

\begin{remark} \label{rem:mp_2P3}
    Observe that for all $k \geq 6$, $P_k$ has a subgraph isomorphic to $2P_3$. Furthermore, if $H$ is a disconnected graph which has two components of order at least three, then there is a subgraph isomorphic to $2P_3$ in $H$. 
\end{remark}

Next consider the infinite family of maximal planar graphs $\{B^{(2)}_p \colon p \geq 1\}$ with a $2$-edge-colouring as in Figure~\ref{fig:c3_mp_infty}.

\begin{figure}[htbp!]
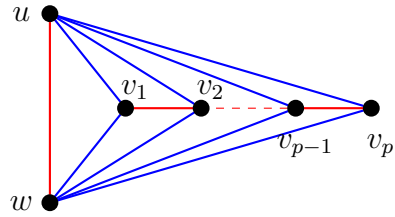

   \ctikzfig{figures/c3_mp_infty}
   \caption{The family of $2$-edge-coloured maximal planar graphs $B^{(2)}_p$ on $p + 2$ vertices, $p \geq 1$.}
   \label{fig:c3_mp_infty}
\end{figure}

For every integer $p \geq 1$, the subgraph induced by the blue edges contains no triangles, whilst the subgraph induced by the red edges contains no cycles. It follows that for every integer $k \geq 3$, $\{C_k, C_3\}$ is avoidable in $\KMP$. By monotonicity and symmetry, we have the following result.

\begin{lemma} \label{lem:mp_avoid_const_B}
   For every integer $k \geq 3$, $\{C_3, C_k\}$ and $\{C_k, C_3\}$ are both avoidable in $\KMP$. Furthermore, if $H$ and $F$ are two graphs such that one has a subgraph isomorphic to $C_3$ and the other has a subgraph isomorphic to $C_k$ for some $k \geq 3$, then $\{H, F\}$ is avoidable in $\KMP$.
\end{lemma}

Similarly, one can verify that for the family $\{B^{(2)}_p \colon p \geq 1\}$ with a $2$-edge-colouring as in Figure~\ref{fig:c3_mp_infty}, the monochromatic red/blue subgraphs contain no copy of~$S_{2,2,1}$.

\begin{lemma} \label{lem:mp_avoid_const_S221}
   $\{S_{2,2,1}, S_{2,2,1}\}$ is avoidable in $\KMP$. Furthermore, if $H, F$ are two graphs such that both have a subgraph isomorphic to $S_{2,2,1}$, then $\{H, F\}$ is avoidable in $\KMP$.
\end{lemma}

\subsection{The diagonal case: unavoidable \texorpdfstring{$\{H, H\}$}{(H,H)} in \texorpdfstring{$\KMP$}{KMP} for connected $H$}

In this section we will prove the complete characterisation given in Theorem \ref{thm:diag_H_conn_mp} of the connected graphs $H$ such that $\{H, H\}$ is unavoidable in $\KMP$. We begin by noting the following immediate consequence of Lemmas \ref{lem:mp_avoid_const_A} and \ref{lem:mp_avoid_const_B}, together with Remark \ref{rem:mp_2P3}.

\begin{proposition} \label{prop:forst_mp}
    If $H$ is a graph such that $\{H, H\}$ is unavoidable in $\KMP$, then $H$ is a forest such that: (i) at most one component has order greater than three and (ii) every component has diameter at most five.
\end{proposition}

We first prove that $\{K_{1, 4}, K_{1, 4}\}$ is avoidable in $\KMP$, before proceed to show that $\{P_5, P_5\}$ and $\{S_{2,1,1}, S_{2,1,1}\}$ are both unavoidable in $\KMP$.

\begin{proposition} \label{prop:k1_4}
    $\{K_{1, 4}, K_{1, 4}\}$ is avoidable in $\KMP$.
\end{proposition}

\begin{proof}
It suffices to construct an infinite sequence $G_i$ of maximal planar graphs with $\Delta = 6$ and edge-chromatic number $\Delta$: For any $6$-edge-colouring of each $G_i$ using the colours $[6]$, if we re-colour all the edges with a colour in $\{1, 2, 3\}$ as blue and all the edges with colour in $\{4, 5, 6\}$ as red, then after re-colouring every vertex of degree at least $4$ has at most three edges of the same colour and hence there is no monochromatic $K_{1, 4}$.

\begin{figure}[htbp!]
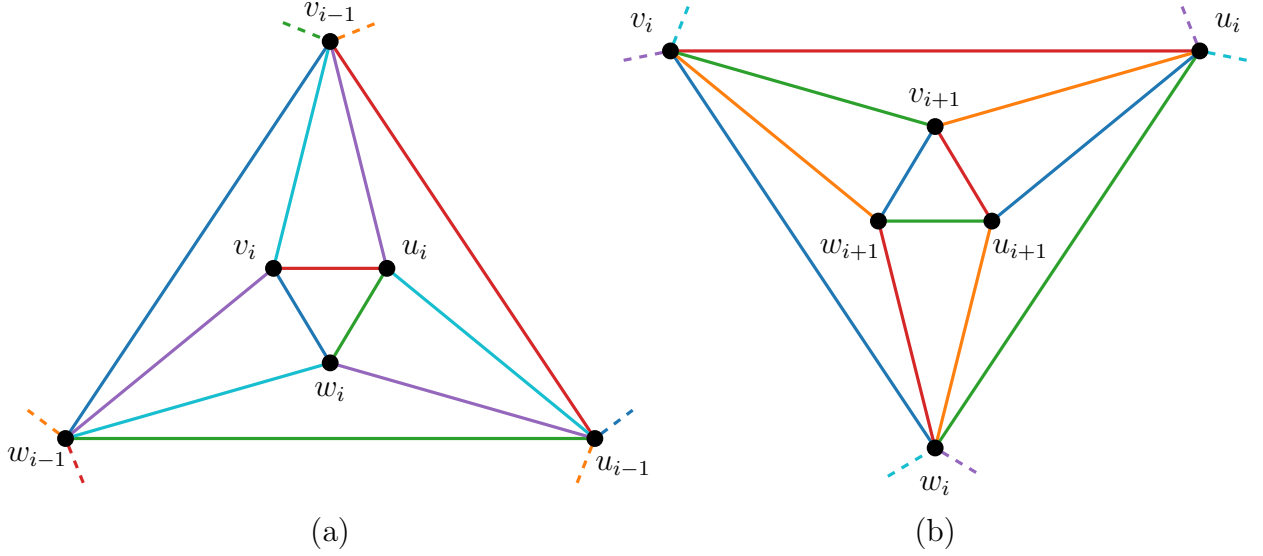

	   \ctikzfig{figures/k1_4_inf_6colouring}
	   \caption{Construction of an infinite family of maximal planar graphs with $\Delta = 6$ and edge-chromatic number $\Delta$ (Vizing Class 1).}
       \label{fig:k1_4_inf_6colouring}
\end{figure}

We next describe one such infinite family of graphs. Consider a triangle formed by the vertices $\{v_0, u_0, w_0\}$, and add another triangle formed by the vertices $\{v_1, u_1, w_1\}$ inside the initial triangle formed by $\{v_0, u_0, w_0\}$; connect these two triangles as shown in Figure~\ref{fig:k1_4_inf_6colouring}~(a) and colour the edges accordingly. Next, inside the triangle formed by $\{v_1, u_1, w_1\}$, add another triangle formed by the vertices $\{v_2, u_2, w_2\}$ and connect the two triangles as shown in Figure~\ref{fig:k1_4_inf_6colouring}~(b) and colour the edges accordingly.

Let the resulting coloured graph be called $G_2$. Clearly it is maximal planar, has $\Delta = 6$ (namely the vertices $v_1, u_1, w_1$), and edge-chromatic number $\Delta$. 

Let $i \geq 3$ be an integer. Observe that, any arrangement as in Figure~\ref{fig:k1_4_inf_6colouring}~(a) can be extended by adding a triangle on the inner-most face, then connecting it and colouring all new edges as in Figure~\ref{fig:k1_4_inf_6colouring}~(b). 

Furthermore, ignoring the inner triangle formed by the vertices $\{v_i, u_i, w_i\}$ in Figure~\ref{fig:k1_4_inf_6colouring}~(a) and the edges incident to them, the edges incident to $\{v_{i+1}, u_{i+1}, w_{i+1}\}$ in Figure~\ref{fig:k1_4_inf_6colouring}~(b) share the same colouring as the edges incident to $\{v_{i-1}, u_{i-1}, w_{i-1}\}$ in Figure~\ref{fig:k1_4_inf_6colouring}~(a). Hence any arrangement as in Figure~\ref{fig:k1_4_inf_6colouring}~(b) can also be extended by adding a triangle on the inner-most face, then connecting it and colouring all new edges as in Figure~\ref{fig:k1_4_inf_6colouring}~(a).

Every such extension results in another maximal planar graph with $\Delta = 6$ and edge-chromatic number $\Delta$. Therefore, if $i$ is odd then let $G_i$ be the extension of $G_{i - 1}$ by adding a triangle on the inner-most face as in Figure~\ref{fig:k1_4_inf_6colouring}~(a); else if $i$ is even then let $G_i$ be the extension of $G_{i - 1}$ by adding a triangle on the inner-most face as in Figure~\ref{fig:k1_4_inf_6colouring}~(b). The resulting family of graphs satisfies all desired properties.
\end{proof}

\begin{remark}
    Vizing proved in \cite{Vizing1965} that every planar graph with $\Delta \geq 8$ has edge-chromatic number $\Delta$ (i.e. is of Vizing Class 1), and further conjectured that every planar graph with $6 \leq \Delta \leq 7$ is of Vizing Class 1. The conjecture has been proved by Zhang for $\Delta = 7$ in \cite{Zhang2000}, but it remains wide open for $\Delta = 6$. This necessitates our construction in the proof of Proposition \ref{prop:k1_4}. 
\end{remark}

In order to prove that $\{P_5, P_5\}$ is unavoidable in $\KMP$, we use the following lemma.

\begin{lemma} \label{lem:wk_mono_p5}
Let $k \geq 5$ be an integer. Every $2$-edge-colouring of the wheel graph $W_k$ has a monochromatic $P_5$.
\end{lemma}

\begin{proof}
    We will use red/blue for the $2$-edge-colouring. The case when $k = 5$ can be readily verified by case analysis. We will prove the case when $k \geq 6$. Let $v_1, \dots, v_k$ be the vertices on the outer cycle of $W_k$ and let $c$ be the centre vertex. 

    If there is a monochromatic path on five or more vertices on the outer cycle of $W_k$, then we are done. Suppose otherwise, i.e. every monochromatic path on the outer cycle has at most four vertices.

    \case{1} We first consider the case when the longest monochromatic path on the cycle is on four vertices, and is coloured say blue. Then there must be some $i$ such that the edge $\{v_{i-1}, v_i\}$ is red, $v_i, v_{i+1}, v_{i+2}, v_{i+3}$ is a blue path and the edge $\{v_{i+3}, v_{i+4}\}$ is red, such that the two red edges are disjoint (since $k \geq 6$). If at least one of the edges $\{c, v_i\}$ or $\{c, v_{i-1}\}$ is red, and at least one of the edges of the edges $\{c, v_{i+3}\}$ or $\{c, v_{i+4}\}$ is red, then we have a red $P_5$ (e.g. if $\{c, v_{i-1}\}$ and $\{c, v_{i+3}\}$ are red, then $v_i, v_{i-1}, c, v_{i+3}, v_{i+4}$ is a red $P_5$). Otherwise, either $\{c, v_{i-1}\}$ and $\{c, v_i\}$ are blue and we have a blue $P_5$ (e.g. the path $c, v_i, v_{i+1}, v_{i+2}, v_{i+3}$), or $\{c, v_{i+3}\}$ and $\{c, v_{i+4}\}$ are blue and we have a blue $P_5$.

    \case{2} We next consider the case when the longest monochromatic path on the cycle is on three vertices, and is coloured say blue. Then there must be some $i$ such that the edge $\{v_{i-1}, v_i\}$ is red, $v_i, v_{i+1}, v_{i+2}$ is a blue path and the edge $\{v_{i+2}, v_{i+3}\}$ is red, such that the two red edges are disjoint (since $k \geq 6$). If at least one of the edges $\{c, v_i\}$ or $\{c, v_{i-1}\}$ is red, and at least one of the edges of the edges $\{c, v_{i+2}\}$ or $\{c, v_{i+3}\}$ is red, then we have a red $P_5$ (e.g. if $\{c, v_{i-1}\}$ and $\{c, v_{i+2}\}$ are red, then $v_i, v_{i-1}, c, v_{i+2}, v_{i+3}$ is a red $P_5$). Otherwise, either $\{c, v_{i-1}\}$ and $\{c, v_i\}$ are blue and we have a blue $P_5$ (namely the path $v_{i-1}, c, v_i, v_{i+1}, v_{i+2}$), or $\{c, v_{i+2}\}$ and $\{c, v_{i+3}\}$ are blue and we have a blue $P_5$.

    \case{3} Otherwise the longest monochromatic path on the cycle is on three vertices, i.e. the edges on the outer cycle of $W_k$ are coloured alternating red and blue, where $k$ must be even. By the pigeon-hole principle, at least $k/2$ edges incident to the centre $c$ have the same colour, say red.

    \case{3.1} Suppose two red edges from the centre are incident to $v_i$ and $v_{i+1}$. If the colour of the edge $\{v_i,v_i+1\}$ is blue, then the edges $\{v_{i-1}, v_i\}$ and $\{v_{i+1}, v_{i+2}\}$ are red, and hence $v_{i-1}, v_i, c, v_{i+1}, v_{i+2}$ is a red $P_5$. Otherwise, the edge $\{v_i,v_i+1\}$ is red. Since $k \geq 6$ then $k/2 \geq 3$ and the centre $c$ has a third red edge incident to some $v_j$ distinct from $v_i$ and $v_{i + 1}$; then one of the edges incident with $v_j$ is red and we get red $P_5$ (e.g. if $\{v_{j - 1}, v_j\}$ is red then $v_{j - 1}, v_j, c, v_i, v_{i+1}$ is a red $P_5$).
 
    \case{3.2} Otherwise suppose that no two red edges from the centre $c$ are incident to consecutive vertices of the rim. But this is possible only if the only if the edges incident to $c$ are coloured alternating red and blue. Without loss of generality, suppose that the edges $\{c, v_1\}$, $\{c, v_3\}$ and $\{c, v_5\}$ are coloured red. Then the alternating colouring of the outer cycle of $W_k$ forces a red edge incident to $v_1$ (either $\{v_k, v_1\}$ or $\{v_1, v_2\}$), and a red edge incident to $v_5$ (either $\{v_4, v_5\}$ or $\{v_5, v_6\}$). Since $k \geq 6$ these two red edges on the outer cycle incident to $v_1$ and $v_5$ are distinct, and together with the edges $\{c, v_1\}$ and $\{c, v_5\}$ for a red $P_5$.

    The result follows.
\end{proof}

We are now in a position to prove that $\{P_5, P_5\}$ is unavoidable in $\KMP$. We will, in fact, by means of the `rim-trick' (Theorem \ref{thm:rim-trick-cycle}) prove that $\{P_5 \cup t K_2, P_5 \cup t K_2\}$ is unavoidable in $\KMP$ for every integer $t \geq 0$.

\begin{theorem} \label{thm:mp_p5}
    For every integer $t \geq 0$, $\{P_5 \cup t K_2, P_5 \cup t K_2\}$ is unavoidable in $\KMP$. In particular, $\MP{P_5}{P_5} = 6$ and $\MP{P_5 \cup t K_2}{P_5 \cup t K_2} \leq (4t^2 + 16t + 26)^{\log_2 3}$.
\end{theorem}

\begin{proof}
    Figure \ref{fig:appendixB4} in Appendix \ref{sec:small_cexs} illustrates a maximal planar graph on $n = 5$ vertices with no monochromatic $P_5$. By an exhaustive computer search, as outlined in Appendix \ref{sec:code}, one can verify that $n = 6$ is the smallest $n$ such that every $2$-edge-colouring of every maximal planar graph on $n$ vertices contains a monochromatic copy of $P_5$. 
    
    Hence~$\MP{P_5}{P_5} \geq 6$. Let $G$ be a maximal planar graph on $n \geq 7$ vertices. Then $G$ has a vertex $v$ of degree at least $5$, and the subgraph induced by the closed neighbourhood $N[v]$ is isomorphic to the wheel graph $W_k$ for some $k \geq 5$. By Lemma \ref{lem:wk_mono_p5}, every $2$-edge-colouring of $N[v]$ contains a monochromatic $P_5$, and hence every $2$-edge-colouring of $G$ contains a monochromatic $P_5$. Hence $\{P_5, P_5\}$ is unavoidable in $\KMP$, with $\MP{P_5}{P_5} = 6$.

    Fix an integer $t \geq 1$. We next prove that $\{P_5 \cup t K_2, P_5\}$ is unavoidable in $\KMP$. Let $G$ be a maximal planar graph on $n \geq (8t + 26)^{\log_2 3}$ vertices. By Theorem \ref{thm:mp_longest_cycle}, $G$ contains a cycle of length $\geq 8t + 26 = (2t-1+|P_5|)(|P_5|-1) + 2|P_5|$. Moreover, as $P_5$ is a linear forest and $G$ either has a red $P_5$ or a blue $P_5$, this means we can apply Theorem \ref{thm:rim-trick-cycle} to conclude that every red/blue edge-colouring of $G$ graph has a red $P_5 \cup t K_2$ or a blue $P_5$.
    
    Now, let $H$ be a maximal planar graph on $n' \geq (4t^2 + 16t + 26)^{\log_2 3} > (8t + 26)^{\log_2 3}$ vertices. By the previous remark with the colours flipped, $H$ either has a red $P_5$ or a blue $P_5 \cup t K_2$. Also, by Theorem \ref{thm:mp_longest_cycle}, $H$ contains a cycle of length  $\geq 4t^2 + 16t + 26 = (2t-1+|P_5|)(|P_5 \cup t K_2| - 1) + 2 |P_5|$. Noting that $P_5 \cup t K_2$ is a linear forest, we can therefore apply Theorem \ref{thm:rim-trick-cycle} again to obtain that every red/blue colouring of $G$ graph has a monochromatic $P_5 \cup t K_2$.
\end{proof}

\begin{remark}
    By monotonicity, if $H$ and $F$ are both subgraphs of $P_5 \cup t K_2$ for some integer $t \geq 0$, then $\{H, F\}$ are unavoidable in $\KMP$ as a consequence of Theorem \ref{thm:mp_p5}.
\end{remark}

Lastly, we prove that $\{S_{2,1,1}, S_{2,1,1}\}$ is also unavoidable in $\KMP$.

\begin{theorem} \label{thm:mp_fork}
    $\{S_{2,1,1}, S_{2,1,1}\}$ is unavoidable in $\KMP$ and $\MP{S_{2,1,1}}{S_{2,1,1}} = 9$.
\end{theorem}

\begin{proof}
    Let $G$ be a maximal planar graph with a vertex $c$ of degree at least $6$; then the subgraph induced by the closed neighbourhood $N[c]$ is isomorphic to the wheel graph $W_k$ for some $k \geq 6$. Furthermore, at least three neighbours of $c$ must have degree at least $4$. Let $v_1, \dots, v_k$ be the neighbours of $c$. Consider a red/blue edge-colouring of $G$; we will outline three classes of colourings of $G$. 

    \begin{figure}[ht!]
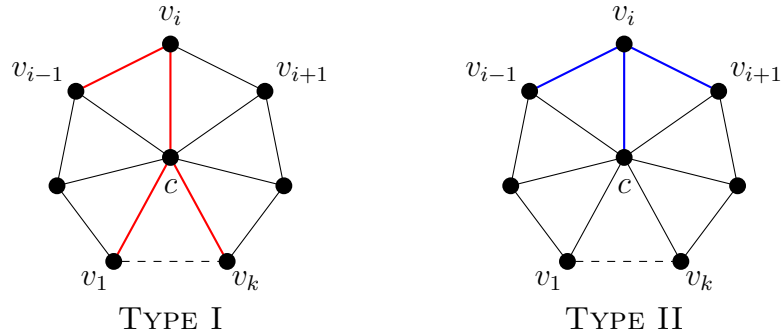

	   \ctikzfig{figures/f5_fig1}
	   \caption{Illustration of $2$-edge-colourings of $N[c]$ of \textsc{Type I} and \textsc{Type II}, where black (dashed) edges may be coloured either red or blue.}
       \label{fig:f5_fig1}
    \end{figure}

    A colouring of \textsc{Type I} is one such that there exists a monochromatic embedding of $S_{2,1,1}$ in the subgraph induced by $N[c]$ where the vertex of degree $3$ in $S_{2,1,1}$ is mapped to $c$, as in Figure~\ref{fig:f5_fig1}. In this case we are immediately done.

    A colouring of \textsc{Type II} is one such that there exists some $i$ where all edges incident to $v_i$ in the subgraph induced by $N[c]$ are all coloured the same, as in Figure~\ref{fig:f5_fig1}. Note that the colourings of \textsc{Type I} and \textsc{Type II} are not necessarily disjoint.

    We will show that \textsc{Type II} colourings have a monochromatic $S_{2,1,1}$. Without loss of generality, let all the edges incident to $v_i$ in the subgraph induced by $N[c]$ be coloured~blue. 
    
    \case{A.1} Suppose at least one of the edges $\{v_{i+1}, v_{i+2}\}$, $\{v_{i-1}, v_{i-2}\}$ or $\{c, v_j\}$ for some $j \notin \{i-1, i+1\}$ is blue. Then there is a blue $S_{2,1,1}$ in this colouring of $G$.
    
    \case{A.2} Otherwise suppose that all of the edges $\{v_{i+1}, v_{i+2}\}$, $\{v_{i-1}, v_{i-2}\}$ and $\{c, v_j\}$ for all $j \notin \{i-1, i+1\}$ are coloured red, as in Figure~\ref{fig:f5_fig2}. Since $\deg(c) \geq 6$ then there is some vertex $v_j \in N(c)$ such that $v_j \notin \{v_{i-2}, v_{i-1}, v_i, v_{i+1}, v_{i+2}\}$. Then the edges $\{v_{i+1}, v_{i+2}\}$, $\{v_{i+2}, c\}$, $\{c, v_{i-2}\}$, $\{c, v_j\}$ form a red $S_{2,1,1}$ in $G$. 

    \begin{figure}[ht!]
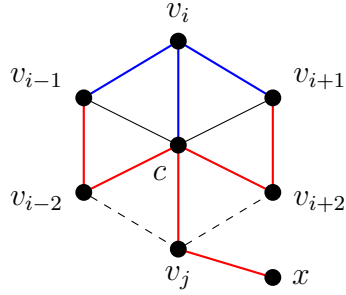

	   \ctikzfig{figures/f5_fig2}
	   \caption{A $2$-edge-colouring of \textsc{Type II} satisfying Case A.2, with a red $S_{2,1,1}$, where black (dashed) edges may be coloured either red or blue.}
       \label{fig:f5_fig2}
    \end{figure}
    
    Hence all \textsc{Type II} colourings have a monochromatic $S_{2,1,1}$ in $G$. Lastly, a colouring is of \textsc{Type III} if it is neither \textsc{Type I} or \textsc{Type II}. We will next show that \textsc{Type III} colourings have a monochromatic $S_{2,1,1}$. 

    If there are at least three edges incident to $c$ which are coloured red and at least three edges incident to $c$ which are coloured blue, then the colouring is of \textsc{Type I}: There exists $i$ such that $\{v_i, c\}$ and $\{v_{i+1}, c\}$ are coloured differently, say red and blue respectively (without loss of generality). If the edge $\{v_i, v_{i+1}\}$ is coloured blue, then $\{v_i, v_{i+1}\}$ and $\{v_i, c\}$ along with two other blue edges incident to $c$ give a blue $S_{2,1,1}$. Otherwise, if the edge $\{v_i, v_{i+1}\}$ is coloured red, then $\{v_i, v_{i+1}\}$ and $\{v_{i+1}, c\}$ along with two other red edges incident to $c$ give a red $S_{2,1,1}$. Figure~\ref{fig:f5_fig3} illustrates this.

    \begin{figure}[h!]
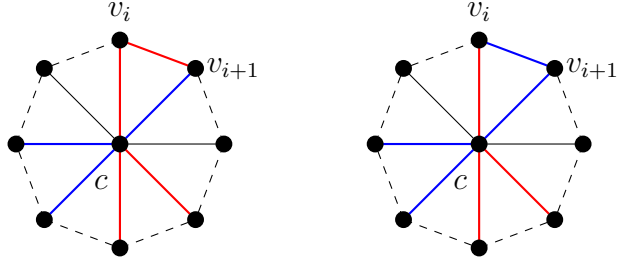

	   \ctikzfig{figures/f5_fig3}
	   \caption{A red/blue $S_{2,1,1}$ in any colouring with at least three red edges and at least three blue edges incident to $c$, where black (dashed) edges may be coloured either red or blue.}
       \label{fig:f5_fig3}
    \end{figure}

    Hence, in a \textsc{Type III} colouring we have $k = \deg(c) \geq 7$ and (without loss of generality) there are at most two red edges incident to $c$ and at least $k - 3 \geq 4$ blue edges incident to $c$. If the number of red edges incident to $c$ is at most one, then since we have a \textsc{Type III} colouring, the edges on the cycle $v_1, \dots, v_k$ are all red (otherwise there exists some $v_i$ such that $\{c, v_i\}$ is blue and at least one of the edges $\{v_{i-1}, v_i\}$, $\{v_i, v_{i+1}\}$ is also blue).

    Otherwise suppose that the number of red edges incident to $c$ is two. Suppose that for some $i$, the edge $\{c, v_i\}$ is red and the edges $\{c, v_{i-1}\}$, $\{c, v_{i+1}\}$ are blue. Since there are at least four edges incident to $c$ coloured blue and the colouring is not of \textsc{Type I}, the edges $\{v_{i-1}, v_i\}$, $\{v_i, v_{i+1}\}$ are red. But such a colouring is not possible, as the edges $\{c, v_i\}$, $\{v_{i-1}, v_i\}$, $\{v_i, v_{i+1}\}$ are all red and hence the colouring would be of \textsc{Type II}. Hence all the red edges incident to $c$ must correspond to consecutive vertices on the cycle $v_1, \dots, v_k$.

    Hence there exists some $i$ such that $\{c, v_i\}$ and $\{c, v_{i+1}\}$ are red, and all other edges incident to $c$ are blue. Since the colouring is neither of \textsc{Type I} or \textsc{Type II}, the only possible colouring of the cycle $v_1, \dots, v_k$ is $\{v_i, v_{i+1}\}$ coloured blue and all other edges red. 

    Collating everything together, we obtain that in a \textsc{Type III} colouring we have $k = \deg(c) \geq 7$, and, without loss of generality, either there is most one red edge incident to $c$ and the edges on the cycle $v_1, \dots, v_k$ are coloured blue, or there are exists some $i$ such that $\{c, v_i\}$ and $\{c, v_{i+1}\}$ are the only red edges incident to $c$ and $\{v_i, v_{i+1}\}$ is the only blue edge on the cycle $v_1, \dots, v_k$. Since there are at least three vertices in $N(c)$ of degree at least four, then there exists some $j$ such that $\{c, v_j\}$ is blue, $v_j$ has degree at least four, and the edges $\{v_{j-2}, v_{j-1}\}$, $\{v_{j-1}, v_j\}$ and $\{v_j, v_{j+1}\}$ are all red (since $k \geq 7$ and there is at most one blue edge on the cycle $v_1, \dots, v_k$).

    \begin{figure}[htbp!]
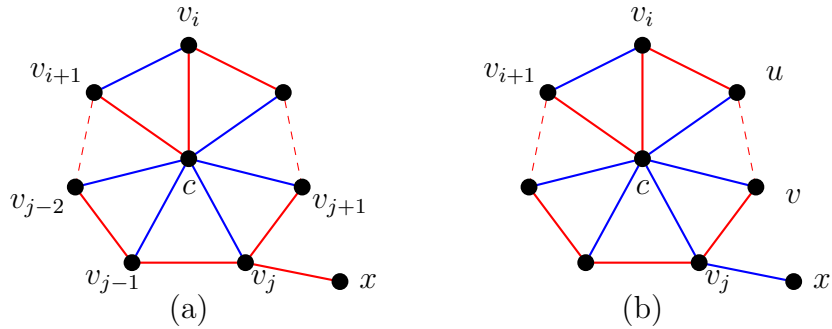

	   \ctikzfig{figures/f5_fig5}
	   \caption{A $2$-edge-colouring of \textsc{Type III} satisfying Case B.2.}
       \label{fig:f5_fig5}
    \end{figure}

    Then there is some $x \notin N[c]$ such that $x$ is adjacent to $v_j$. If $\{v_j, x\}$ is red then there is a red $S_{2,1,1}$, namely by the edges $\{v_{j-2}, v_{j-1}\}$, $\{v_{j-1}, v_j\}$, $\{v_j, v_{j+1}\}$ and $\{v_j, x\}$, as in Figure~\ref{fig:f5_fig5}~(a). Else if $\{v_j, x\}$ is blue, then since there exists two vertices $u, v \in N(c)$ distinct from $v_j$ such that the edges $\{u, c\}$ and $\{v, c\}$ are also blue, together with the edges $\{c, v_j\}$ and $\{v_j, x\}$ they give a blue $S_{2,1,1}$, as in Figure~\ref{fig:f5_fig5}~(b).

    Consequently, every \textsc{Type III} colouring has a monochromatic $S_{2,1,1}$. It follows that every $2$-edge-colouring of a maximal planar graph $G$ with $\Delta \geq 6$ has a monochromatic $S_{2,1,1}$. Since every maximal planar graph on $n \geq 13$ vertices has a vertex of degree at least six, $\{S_{2,1,1}, S_{2,1,1}\}$ is unavoidable in $\KMP$ with $\MP{S_{2,1,1}}{S_{2,1,1}} \leq 13$. 

    Figure \ref{fig:appendixB3} in Appendix \ref{sec:small_cexs} illustrates a maximal planar graph on $n = 8$ vertices with no monochromatic $S_{2,1,1}$. By an exhaustive computer search, as outlined in Appendix \ref{sec:code}, one can verify that $n = 9$ is the smallest $n$ such that every $2$-edge-colouring of every maximal planar graph on $n$ vertices contains a monochromatic copy of $S_{2,1,1}$. Hence~$\MP{S_{2,1,1}}{S_{2,1,1}} \geq 9$. 
    
    To show that $\MP{S_{2,1,1}}{S_{2,1,1}} = 9$, it remains to show that every $2$-edge-colouring of every maximal planar graph with $\Delta \leq 5$ and $10 \leq n \leq 12$ vertices has a monochromatic $S_{2,1,1}$. Let $G$ be a maximal planar graph on $n = 10$ vertices with $\Delta \leq 5$; if $\delta = 3$ then there is a vertex $v$ in $G$ such that $G - v$ is a maximal planar graph on $9$ vertices. In particular, every $2$-edge-colouring of $G - v$ must have a monochromatic~$S_{2,1,1}$. 

    \begin{figure}[htbp!]
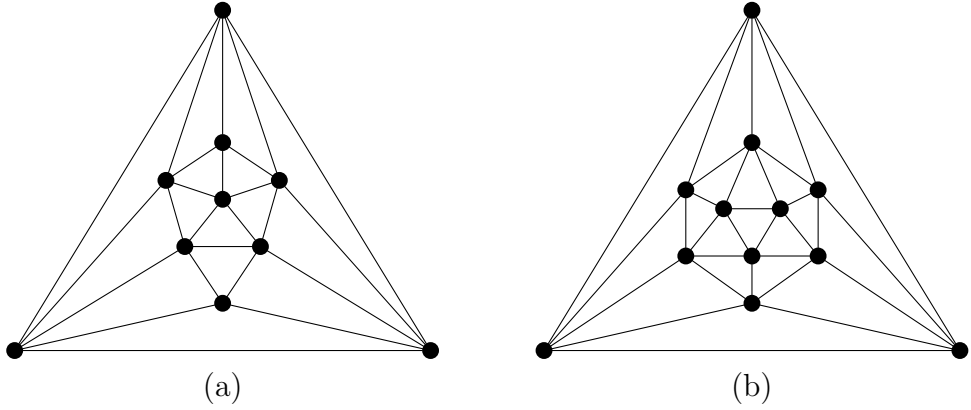

	   \ctikzfig{figures/mp_s211_n10_n12_exceptions}
	   \caption{The two maximal planar graphs on $10 \leq n \leq 12$ vertices with $4 \leq \delta \leq \Delta \leq 5$.}
       \label{fig:f5_fig6}
    \end{figure}
    
    It remains to consider those maximal planar graphs on $n = 10$ vertices with $4 \leq \delta \leq \Delta \leq 5$. There is only one such graph, illustrated in Figure~\ref{fig:f5_fig6}~(a), for which by exhaustive computer search we can verify that every $2$-edge-colouring has a monochromatic $S_{2,1,1}$. Having verified the case when $n = 10$, we repeat the same analysis for $n = 11$ and $n = 12$. There is only one other graph with $4 \leq \delta \leq \Delta \leq 5$ in these cases, illustrated in Figure~\ref{fig:f5_fig6}~(b). Once again, by exhaustive computer search we can verify that every $2$-edge-colouring of this graph has a monochromatic $S_{2,1,1}$. 
    
    Combining everything together, we obtain that $\MP{S_{2,1,1}}{S_{2,1,1}} = 9$.
\end{proof}

We are now in a position to prove Theorem \ref{thm:diag_H_conn_mp}, which we restate for convenience.

\MPDiagonal*

\begin{proof}
    As a consequence of Proposition \ref{prop:forst_mp}, any connected graph $H$ such that $\{H, H\}$ is unavoidable in $\KMP$ must be a tree, with every component having diameter at most five. From planar T{\'u}ran, we established in Corollary \ref{cor:mp_small_trees} that for every tree $H$ on at most four vertices, $\{H, H\}$ is unavoidable in $\KMP$. Therefore it remains to consider trees on $n \geq 5$ vertices. 

    On $n = 5$ vertices, there are three trees to consider: $P_5$, $S_{2,1,1}$ and $K_{1, 4}$. By Theorems \ref{thm:mp_p5} and \ref{thm:mp_fork}, we have that $\{P_5, P_5\}$ and $\{S_{2,1,1}, S_{2,1,1}\}$ are both unavoidable in $\KMP$, whilst by Proposition \ref{prop:k1_4} we have that $\{K_{1,4}, K_{1,4}\}$ is avoidable in $\KMP$.

    \begin{figure}[htbp!]
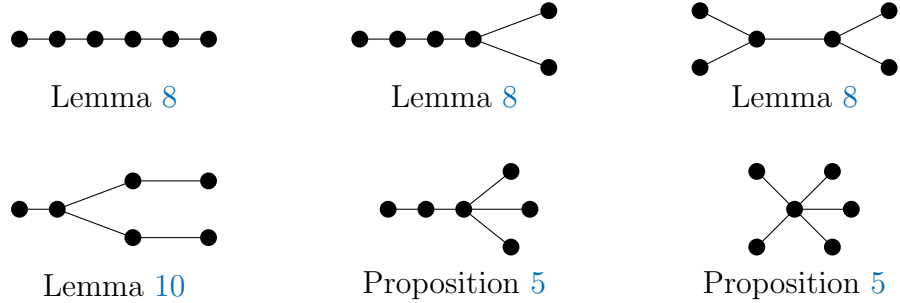

	   \ctikzfig{figures/trees_n6}
	   \caption{All trees on $n = 6$ vertices are avoidable in $\KMP$.}
       \label{fig:trees_n6}
    \end{figure}

    Next we show that all trees on $n \geq 6$ are avoidable in $\KMP$. Figure~\ref{fig:trees_n6} illustrates all trees on $n = 6$ vertices, along with the corresponding lemmas showing they are avoidable $\KMP$. By monotonicity, all trees on $n \geq 6$ vertices are avoidable in $\KMP$. The result follows.
\end{proof}

\section{Open problems}
\label{sec:problems}

We conclude with a number of open problems arising from this work, beginning with the following problem on which outerplanar graphs $H$ are such that $\{H, K_2\}$ is unavoidable in $\KMOP$, as further discussed in Remark \ref{rem:mop_H_K2_problem}.

\begin{problem}
    Characterise all outerplanar graphs $H$ for which there exists an integer $N(H) \geq |H|$ such that for all $n \geq N(H)$, every maximal outerplanar graph on $n$ vertices contains a copy of $H$.
\end{problem}
 
We have the following problem for the non-connected diagonal case in $\KMP$.

\begin{problem}
    Completely determine all disconnected planar graphs $H$ such that $\{H, H\}$ is unavoidable in $\KMP$.
\end{problem}

More ambitious is the complete characterisation of all unavoidable pairs $\{H, F\}$ in $\KMP$, which for the pairs $\{H, K_2\}$ includes the complete characterisation of all planar graphs $H$ for which there exists an integer $N(H) \geq |H|$ such that for all $n \geq N(H)$, every maximal planar graph on $n$ vertices contains a copy of $H$.

\begin{problem}
    Completely determine all pairs of graphs which are unavoidable in $\KMP$.
\end{problem}

\bibliographystyle{plain}
\bibliography{mp_refs}

\appendix
\newpage

\section{Small instances of graphs with avoidable $2$-edge-colourings for unavoidable pairs}
\label{sec:small_cexs}

\vspace*{-2mm}

\begin{figure}[H]
   \ctikzfig{figures/mop_s221_p3_n6}
   \caption{Example of a maximal outerplanar graph on $n = 6$ vertices with a red/blue edge-colouring having no red $S_{2,2,1}$ and no blue $P_3$.}
   \label{fig:appendixB1}
\end{figure}

\vspace*{-5mm}

\begin{figure}[H]
   \ctikzfig{figures/mop_p5_p4_n8}
   \caption{Example of a maximal outerplanar graph on $n = 8$ vertices with a red/blue edge-colouring having no red $P_5$ and no blue $P_4$.}
   \label{fig:appendixB2}
\end{figure}

\vspace*{-5mm}

\begin{figure}[H]
   \ctikzfig{figures/mp_p5_n5}
   \caption{Example of a maximal planar graph on $n = 5$ vertices with a red/blue edge-colouring having no monochromatic $P_5$.}
   \label{fig:appendixB4}
\end{figure}

\vspace*{-5mm}

\begin{figure}[H]
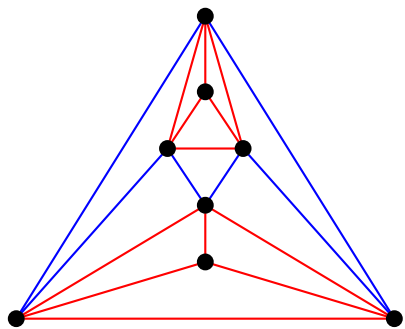

   \ctikzfig{figures/mp_s211_n8}
   \caption{Example of a maximal planar graph on $n = 8$ vertices with a red/blue edge-colouring having no monochromatic $S_{2,1,1}$.}
   \label{fig:appendixB3}
\end{figure}

\newpage
\section{Mathematica code}
\label{sec:code}

We use the following Mathematica \cite{Mathematica} code to exhaustively check whether a given pair $\{H, F\}$ is (un)avoidable in $\KMP$ or $\KMOP$ up to a (small) given number of vertex $n$, typically $n \leq 9$. Small instances of maximal (outer)planar graphs were generated using \texttt{plantri} \cite{plantri}.

\begin{mmaCell}[
  moredefined={RedBlueColouring,VertexList,EdgeList,Graph},
  local={vertices,edges,colouring,redEdges,blueEdges},
  pattern={g_,g,i_,i}
]{Code}
(* The following function generates the i^th red/blue edge 
   colouring of an input graph g, and for the colouring 
   returns a pair of graphs {R, B} having the same vertex
   set as g and monochromatic edge sets from the colouring *)
RedBlueColouring[g_, i_] := 
Module[{vertices, edges, colouring, redEdges, blueEdges},
   vertices = VertexList[g];
   edges = EdgeList[g];
   colouring = IntegerDigits[i - 1, 2, Length[edges]];
   redEdges = Pick[edges, colouring, 1];
   blueEdges = Pick[edges, colouring, 0];
   Return[{Graph[vertices, redEdges],
           Graph[vertices, blueEdges]}]
];
\end{mmaCell}

\begin{mmaCell}[
  moredefined={IsomorphicSubgraphQ,CheckPairs},
  pattern={targetPair_List,targetPair,hostPair_List,hostPair}
]{Code}
(* The following function considers a `target pair' {H, F}
   and a `host pair' {R, B} and checks if R has a subgraph
   isomorphic to H or if B has a subgraph isomorphic to F *)
CheckPairs[targetPair_List, hostPair_List] := 
  IsomorphicSubgraphQ[targetPair[[1]], hostPair[[1]]] || 
   IsomorphicSubgraphQ[targetPair[[2]], hostPair[[2]]];
\end{mmaCell}

\begin{mmaCell}[
  moredefined={RedBlueColouring,UnavoidablePair,CheckPairs,AllTrue,EdgeList},
  pattern={targetPair_List,targetPair,g_Graph,g},
  local={hostPair,i}
]{Code}
(* The following function considers a graph g and a
   `target pair' {H, F}, and checks if every red/blue
   edge colouring of g has a red subgraph isomorphic
   to H or a blue subgraph isomorphic to F *)
UnavoidablePair[g_Graph, targetPair_List] := Module[{i},
  For[i = 1, i <= 2^Length[EdgeList[g]], i++,
   If[Not[CheckPairs[targetPair, RedBlueColouring[g, i]]], 
    Return[False]]];
  Return[True]
];
\end{mmaCell}

\end{document}